\documentclass[11pt]{amsart}
\usepackage{extarrows}
\usepackage[colorlinks, citecolor=blue, dvipdfm, pagebackref]{hyperref}
\usepackage[all]{xy}

\newtheorem{theorem}{Theorem}[section]
\newtheorem{lemma}[theorem]{Lemma}

\theoremstyle{definition}
\newtheorem{definition}[theorem]{Definition}
\newtheorem{question}[theorem]{Question}
\newtheorem{example}[theorem]{Example}

\newtheorem{proposition}[theorem]{Proposition}
\newtheorem{corollary}[theorem]{Corollary}
\newtheorem{remark}[theorem]{Remark}

\theoremstyle{remark}

\newcommand{\be}{\begin{equation}}
\newcommand{\ee}{\end{equation}}

\numberwithin{equation}{section}

\begin{document}
\title{The $\chi_y$-genus, Chern number inequalities and signature}
\author{Ping Li}
\address{School of Mathematical Sciences, Fudan University, Shanghai 200433, China}

\email{pinglimath@fudan.edu.cn\\
pinglimath@gmail.com}
\author{Yibo Ren}

\address{School of Mathematical Sciences, Fudan University, Shanghai 200433, China}

\email{ybren24@m.fudan.edu.cn}

\thanks{The authors were partially supported by the National
Natural Science Foundation of China (Grant No. 12371066).}




\begin{abstract}
This article has two parts. In the first part we introduce two positivity conditions for the modified $\chi_y$-genus on almost-complex manifolds and show that each of them implies a family of optimal Chern number inequalities. It turns out that many important K\"{a}hler and symplectic manifolds satisfy either of the two positivity conditions, and hence these Chern number inequalities hold true on them. In the second part we focus on the signature, a special value of the $\chi_y$-genus, of symplectic manifolds equipped with symplectic circle actions and give applications. Our results in this part unify and generalize various related results in the existing literature.
\end{abstract}

\maketitle

\section{Introduction}
The $\chi_y$-genus $\chi_y(M)\in\mathbb{Z}[y]$ was introduced by Hirzebruch in his seminal book \cite{Hi66} for projective manifolds $M$. Its coefficients can be expressed in terms of linear combination of Chern numbers via the celebrated Hirzebruch-Riemann-Roch theorem, also established by him in \cite{Hi66}. Since these integers can be interpreted as the indices of some Dolbeault-type elliptic differential operators, the later Atiyah-Singer index theorem
implies that it still holds true for general compact almost-complex manifolds (\cite{AS}). This $\chi_y$-genus has many remarkable properties and applications in related areas. As samples we refer the reader to \cite{Ko70}, \cite{LW}, \cite{Sa}, \cite{Li12}, \cite{Li15}, \cite{Li17}, \cite{Li19}, \cite{LP}, and the references therein.

The \emph{first main purpose} in this article is to introduce two positivity conditions for the (slightly modified) $\chi_y$-genus and show that either of them yields a family of optimal Chern number inequalities. It turns out that manifolds satisfying either of these positivity conditions include rational homogeneous manifolds, smooth toric varieties, Fano contact manifolds, K\"{a}hler hyperbolic manifolds, and symplectic manifolds endowed with some specific symplectic circle actions. The work of this part is partially inspired by \cite{Li19}, where such Chern number inequalities on K\"{a}hler hyperbolic manifolds have been obtained by the first author.

For a special value of $y$, $\chi_y(M)$ is an
important invariant: $\chi_y(M)\big|_{y=1}$ is the signature of $M$. So our \emph{second main purpose} in this article is to focus on the signature of symplectic manifolds admitting symplectic circle actions. Along this line we are able to unify and extend some previous results.

The rest of this article is structured as follows. In Section \ref{section2} we recall Hirzebruch's $\chi_y$-genus, introduce two positivity conditions and state the optimal Chern number inequalities, whose proof shall be given in Section \ref{section3}. Sections \ref{section4} and \ref{section5} are devoted to applications to related K\"{a}hler and symplectic manifolds respectively. In Section \ref{section6} we investigate the signature formula for symplectic manifolds equipped with symplectic circle actions in terms of Novikov numbers and Betti numbers, which is then applied in Section \ref{section7} to get various Betti numbers restrictions on these manifolds.

Before ending this Introduction, we make the convention that all almost-complex, symplectic and K\"{a}hler  manifolds mentioned in the sequel are closed, connected, oriented and of real dimension $2n$, unless otherwise stated. The orientation of such a manifold is always taken the canonical one induced from its (compatible) (almost-)complex structure.

\section{Positivity conditions on the $\chi_y$-genus and Chern number inequalities}\label{section2}
In this section we briefly recall the Hirzebruch $\chi_y$-genus on almost-complex manifolds, introduce on it two positivity conditions (Definition \ref{def}) and state a family of optimal Chern number inequalities (Theorem \ref{main}).

\subsection{The Hirzebruch $\chi_y$-genus}
Let $(M,J)$ be an almost-complex manifold
with almost-complex structure $J$. As
usual we denote by $\bar{\partial}$ the $d$-bar operator which
acts on the complex vector spaces $\Omega^{p,q}(M)$ consisting of $(p,q)$-type complex-valued differential forms on $(M,J)$ in the sense of $J$ (\cite[p.27]{We}), where $0\leq p,q\leq
n$. The choice of an almost-Hermitian metric
on $(M,J)$ enables us to define the formal adjoint
$\bar{\partial}^{\ast}$ of the
$\bar{\partial}$-operator. Then for each $0\leq p\leq n$, we have
the following Dolbeault-type elliptic differential operator $D_p$:
\be\label{GDC}D_p:=\bar{\partial}+\bar{\partial}^{\ast}:~\bigoplus_{\textrm{$q$
even}}\Omega^{p,q}(M)\longrightarrow\bigoplus_{\textrm{$q$
odd}}\Omega^{p,q}(M),\ee whose index is denoted by $\chi^{p}(M)$ in
the notation of Hirzebruch (\cite{Hi66}).
\emph{The Hirzebruch $\chi_{y}$-genus}, denoted by $\chi_{y}(M)$, is the generating
function of these indices $\chi^p(M)$:
$$\chi_{y}(M):=\sum_{p=0}^{n}\chi^{p}(M)\cdot y^{p}.$$

When $J$ is integrable, i.e., $M$ is an $n$-dimensional complex
manifold, which is equivalent to the condition that
$\bar{\partial}^2\equiv0$, the two-step elliptic complex (\ref{GDC})
has  the following resolution, which is the well-known
Dolbeault complex: \be\label{DC}0\rightarrow\Omega^{p,0}(M)
\xrightarrow{\bar{\partial}}\Omega^{p,1}(M)
\xrightarrow{\bar{\partial}}\cdots\xrightarrow{\bar{\partial}}\Omega^{p,n}(M)\rightarrow
0\nonumber\ee and hence
 \be\label{Hodgenumber}\chi^{p}(M)=\sum_{q=0}^{n}(-1)
^{q}\text{dim}_{\mathbb{C}}H^{p,q}
_{\bar{\partial}}(M)=:\sum_{q=0}^{n}(-1)^{q}h^{p,q}(M).\ee Here
$h^{p,q}(M)$ are the Hodge numbers of $M$, which are
the complex dimensions of the corresponding Dolbeault cohomology
groups $H^{p,q} _{\bar{\partial}}(M)$.

For instance, let $\mathbb{P}^n$ be the $n$-dimensional complex projective space, then $\chi^p(\mathbb{P}^n)=(-1)^p$ as $h^{p,p}(\mathbb{P}^n)=1$ and $h^{p,q}(\mathbb{P}^n)=0$ whenever $p\neq q$, and hence \be\label{0}\chi_y(\mathbb{P}^n)=\sum_{p=0}^n(-y)^p.\ee

The general form of the Hirzebruch-Riemann-Roch theorem, which is a
corollary of the Atiyah-Singer index theorem, allows us to compute
$\chi_y(M)$ for an almost-complex manifold $M$ in terms of its Chern numbers as follows (see \cite[\S4]{AS} or \cite[p.61]{HBJ})
\be\label{HRR}\chi_y(M)=
\int_M\prod_{i=1}^n
\frac{x_i(1+ye^{-x_i})}{1-e^{-x_i}},\ee
where $x_1,\ldots,x_n$ are formal Chern roots of $(M,J)$, i.e., the
$i$-th elementary symmetric polynomial of $x_1,\ldots,x_n$ represents $c_i(M)$, the
$i$-th Chern class of $(M,J)$:
 $$c_1(M)=x_1+\cdots+x_n,\quad c_2(M)=\sum_{1\leq i<j\leq n}x_ix_j,\quad\ldots,\quad c_n(M)=x_1x_2\cdots x_n.$$

This $\chi_y(M)$ satisfies the self-reciprocity up to a sign:
\be\label{chiyduality}\chi_y(M)=(-y)^n\cdot\chi_{y^{-1}}(M),\ee
which is equivalent to
$\chi^p=(-1)^n\chi^{n-p}$ for all $p$, and can be derived from (\ref{HRR}). When $J$ is integrable, this fact can also be seen from the Serre duality $h^{p,q}=h^{n-p,n-q}$ via (\ref{Hodgenumber}).

\subsection{The modified version}\label{section2.2}
The \emph{complex genus} in the sense of Hirzebruch is a ring homomorphism from the complex cobordism ring to $\mathbb{Q}$ (\cite[\S 1.8]{HBJ}). It turns out that there is a one-to-one correspondence between genera and normalized formal power series $$Q(x)=1+\sum_{i\geq1}a_ix^i\quad(a_i\in\mathbb{Q}).$$

When viewing $x(1+ye^{-x})/(1-e^{-x})$ arising from (\ref{HRR}) as a formal power series in $x$, its constant term is $1+y$. A simple trick shows that (see \cite[p.62]{HBJ})
\be\chi_y(M)=
\int_M\prod_{i=1}^n
\frac{x_i\big(1+ye^{-x_i(1+y)}\big)}{1-e^{-x_i(1+y)}}.\nonumber\ee
Hence $\chi_y(\cdot)$ is a complex genus whose associated formal power series with a parameter $y$ is
\be\label{power series}Q(y;x)=\frac{x\big(1+ye^{-x(1+y)}\big)}{1-e^{-x(1+y)}}=1+\cdots,\nonumber\ee
which justifies its name and is the \emph{original} definition for almost-complex manifolds by Hirzebruch (\cite[\S10.2]{Hi66}).

For three values of $y$, $\chi_y(M)$ is an
important invariant (\cite[p.62]{HBJ}): $\chi_y(M)\big|_{y=-1}=c_n[M]$ is the Euler
characteristic of $M$, $\chi_y(M)\big|_{y=0}=\chi^0(M)$ the Todd genus of
$M$, and $\chi_y(M)\big|_{y=1}$ the signature of $M$.

Instead of the $\chi_y$-genus, in the sequel we shall consider the $\chi_{-y}$-genus with normalized power series $Q(-y;x)$:
$$\chi_{-y}(M)=\int_M\prod_{i=1}^nQ(-y;x_i)=
\sum_{p=0}^n(-1)^p\chi^p(M)y^p.$$
The advantages for choosing the $\chi_{-y}$-genus shall be clear as the paper progresses \big(recall the formulas (\ref{0}) and (\ref{chiyduality})\big). At this moment we would like to mention a combinatorial ``reason", which has been noticed by Hirzebruch (\cite{Hi92}, \cite[p.12]{Hi08}):
$$\frac{x}{Q(-y;-x)}=\frac{e^{x(1-y)-1}}{1-ye^{x(1-y)}}
=\sum_{i=1}^{\infty}P_i(y)\frac{x^i}{i!},$$
where $P_i(y)$ are the famous \emph{Eulerian polynomials} (\cite[p.22]{St97}). Here $Q(-y;-x)$ is the normalized power series corresponding to the genus $(-1)^n\chi_{-y}(\cdot)$.

\subsection{Positivity conditions and Chern number inequalities}
\begin{definition}\label{def}
An almost-complex manifold $M$ is called \emph{$\chi$-positive} (resp. \emph{signed $\chi$-positive}) if all the coefficients of $\chi_{-y}(M)$ are positive (resp. signed positive), i.e., $(-1)^p\chi^p(M)>0$ (resp. $(-1)^{n+p}\chi^p(M)>0$) for all $0\leq p\leq n$. 
\end{definition}

With these notions understood, we have the following optimal Chern number inequalities for such manifolds.
\begin{theorem}\label{main}
Let $M$ be an almost-complex manifold, which is either $\chi$-positive or signed $\chi$-positive. Take $\epsilon=1$ (resp. $\epsilon=-1$) if $M$ is $\chi$-positive (resp. signed $\chi$-positive). Then $M$ satisfies $[\frac{n}{2}]+1$ optimal Chern number inequalities
\be\label{chernnumberinequality}
\begin{split}
A_i(c_1,\ldots,c_n)[M]&\geq\epsilon^nA_i\Big({n+1\choose 1},\ldots,{n+1\choose n}\Big)\\
&=\epsilon^nA_{i}
(c_1,\ldots,c_n)[\mathbb{P}^n],\qquad0\leq i\leq[\frac{n}{2}],
\end{split}\ee
where $A_i(c_1,\ldots,c_n)$ can be determined by a recursive algorithm and the first three Chern number inequalities read as follows
\begin{eqnarray}\label{firstfew}
\left\{\begin{array}{ll}
\begin{split}
&\epsilon^nc_n[M]\geq n+1,\\
~\\
&\epsilon^n\Big[\frac{n(3n-5)}{2}c_{n}+c_{1}c_{n-1}\Big][M]
\geq2(n-1)n(n+1),\\
~\\
&\epsilon^n\Big[n(15n^{3}-150n^{2}+485n-502)c_{n}+
4(15n^{2}-85n+108)c_{1}c_{n-1}\\
&+8(c_{1}^{2}+3c_{2})c_{n-2}-
8(c_{1}^{3}-3c_{1}c_{2}+3c_{3})c_{n-3}\Big][M]\\
&\geq\epsilon^nA_2\Big({n+1\choose 1},\ldots,{n+1\choose n}\Big).\end{split}
\end{array} \right.
\end{eqnarray}
Moreover, the $i$-th equality case in (\ref{chernnumberinequality}) occurs if and only if
\be\label{equalitycase1}\chi^p(M)=\epsilon^n(-1)^{p},\qquad 2i\leq p\leq n.\ee
\end{theorem}

\section{Proof of Theorem \ref{main}}\label{section3}
In this section we first recall an interesting property
of the $\chi_y$-genus, which the first author calls ``\emph{the $-1$-phenomenon}", and then apply it to prove Theorem \ref{main}.

\subsection{The $-1$-phenomenon}
When $n$ are small, the formulas of $\chi^p$ in terms of rational linear combinations of Chern numbers can be explicitly written down. For example when $n\leq 6$, $\chi^0$ are listed in \cite[p.14]{Hi66} . However for \emph{general} $n$ there are \emph{no explicit} formulas for these $\chi^p$. Nevertheless, as we have mentioned above, $\chi_y(M)\big|_{y=-1}=c_n[M]$. Note that $\chi_y(M)\big|_{y=-1}$ is precisely
the constant term when expanding the polynomial $\chi_y(M)$ at $y=-1$.
Indeed, several independent articles (\cite{NR}, \cite{LW},
\cite{Sa}) have noticed that, when
expanding the right-hand side of (\ref{HRR}) at $y=-1$, its first few coefficients for \emph{general $n$} have \emph{explicit} formulas in
terms of Chern numbers. To be more precise, regarding these coefficients we have the following facts.
\begin{proposition}\label{ch}
Let $K_j(M)$ $(0\leq j\leq n)$ be the coefficients in front of $(y+1)^j$ in the Taylor expansion of $\chi_y(M)$ at $y=-1$, i.e.,
\be\label{chiy-1}\int_M\prod_{i=1}^n\frac{x_i(1+ye^{-x_i})}{1-e^{-x_i}}
=:\sum_{j=0}^nK_j(M)\cdot(y+1)^j.\ee
So each $K_j$ is a rational linear combination of Chern numbers.
Then we have
\begin{enumerate}
\item
any $K_{2i+1}$ is a linear combination of the set $\{K_{2j}~|~0\leq j\leq i\}$ and so \emph{we are only interested in $K_{2i}$} for $0\leq i\leq[\frac{n}{2}]$,


\item
there is a recursive algorithm to determine the formulas $K_{2i}$, and

\item
the first few terms are given by
\begin{eqnarray}\label{firstfewterms}
\left\{ \begin{array}{ll}
K_0=c_{n},\\
~\\
K_1= -\frac{1}{2}nc_{n},\\
~\\
K_2=
\frac{1}{12}\Big[\frac{n(3n-5)}{2}c_{n}+c_{1}c_{n-1}\Big]\\
~\\
K_3=-\frac{1}{24}\Big[\frac{n(n-2)(n-3)}{2}c_{n}+
(n-2)c_{1}c_{n-1}\Big]\\
~\\
\begin{split}
K_4=\frac{1}{5760}\Big[&n(15n^{3}-150n^{2}+485n-502)c_{n}+
4(15n^{2}-85n+108)c_{1}c_{n-1}\\
&+8(c_{1}^{2}+3c_{2})c_{n-2}-8(c_{1}^{3}-3c_{1}c_{2}+3c_{3})c_{n-3}\Big].
\end{split}
\end{array} \right.
\end{eqnarray}
\end{enumerate}
\end{proposition}
\begin{proof}
We refer to \cite[Lemma 2.1]{Li17} for the proof of $(1)$. The recursive algorithm for calculating $K_{j}$ has been described in \cite[p.144]{LW}.
The formulas for $K_j$ up to $j=9$ can be found in \cite[p.141-143]{LW}, \cite[p.145]{Sa} and \cite{De}.
\end{proof}
\begin{remark}
Among (\ref{firstfewterms}) the formula $K_2$ is extremely useful. For instance, Narasimhan-Ramanan applied it to give a topological restriction on some moduli
spaces of stable vector bundles over Riemann surfaces (\cite[p.18]{NR}). Libgober-Wood applied it in \cite{LW} to prove the
uniqueness of the complex structure on K\"{a}hler manifolds of
certain homotopy types, which were further refined in \cite{De} using the same ideas. Salamon applied it to obtain a restriction on the Betti numbers of hyperK\"{a}hler
manifolds (\cite[Coro. 3.4, Thm 4.1]{Sa}). Very recently, Peternell and the first author also applied it to solve a long-standing conjecture of Fujita concerning the uniqueness on smooth K\"{a}hler compactification of contractible complex manifolds (\cite{LP}).
\end{remark}

\subsection{Proof of Theorem \ref{main}}
As in Theorem \ref{main}, let $\epsilon=\pm1$ corresponding to $\chi$-positivity or signed $\chi$-positivity. By (\ref{chiy-1}) we have
\be
\epsilon^n\sum_{j=0}^n(-1)^jK_j(M)(y-1)^j=
\epsilon^n\chi_{-y}(M)=
\epsilon^n\sum_{p=0}^n(-1)^p\chi^p(M)y^p,
\nonumber\ee
which, for each $0\leq j\leq n$, yields
\be\label{1}
\epsilon^n(-1)^jK_j(M)=\sum_{p=j}^n
\epsilon^n(-1)^p\chi^p(M){p\choose j}.
\ee

Due to Definition \ref{def} the condition of $\chi$-positivity or signed $\chi$-positivity implies that
$$\epsilon^n(-1)^p\chi^p(M)\geq1,\quad 0\leq p\leq n.$$
Therefore (\ref{1}) implies that
\be\label{2}\begin{split}
\epsilon^n(-1)^jK_j(M)\geq\sum_{p=j}^n
{p\choose j}=&\frac{\Big[\sum_{p=0}^ny^p\Big]^{(j)}}
{j!}\Big|_{y=1}\\
=&\frac{\Big[\chi_{-y}(\mathbb{P}^n)\Big]^{(j)}}
{j!}\Big|_{y=1}\quad\big(\text{by (\ref{0})}\big)\\
=&(-1)^jK_j(\mathbb{P}^n).
\end{split}
\ee

Define
$$A_i(c_1,\ldots,c_n):=\epsilon^nK_{2i},\qquad 0\leq i\leq[\frac{n}{2}].$$
It follows from (\ref{2}) that
\be
\begin{split}
A_i(c_1,\ldots,c_n)[M]\geq&\epsilon^nA_{i}
(c_1,\ldots,c_n)[\mathbb{P}^n]\\
=&\epsilon^nA_i\Big({n+1\choose 1},\ldots,{n+1\choose n}\Big),\qquad 0\leq i\leq[\frac{n}{2}].\end{split}\nonumber\ee
This establishes the inequalities (\ref{chernnumberinequality}) in Theorem \ref{main}. The characterization of the equality case (\ref{firstfew}) follows from (\ref{1}) and (\ref{2}). The three Chern number inequalities (\ref{firstfew}) then follow from (\ref{firstfewterms}).

\section{Applications to K\"{a}hler manifolds}\label{section4}
We shall apply Theorem \ref{main} in this section to various K\"{a}hler manifolds satisfying either $\chi$-positivity or signed $\chi$-positivity.
\subsection{K\"{a}hler manifolds with $\chi$-positivity}
Our definition for the $\chi$-positivity is motivated by the following
\begin{definition}
A K\"{a}hler manifold $M$ is called of \emph{pure type} if its Hodge numbers $h^{p,q}(M)=0$ whenever $p\neq q$. In other words, $M$ is of pure type if and only if its nonzero Hodge numbers concentrate only on the vertical line of the Hodge diamond (\cite[p.117]{GH}).
\end{definition}

The following facts are straightforward.
\begin{lemma}
A K\"{a}hler manifold $M$ of pure type satisfies
\be(-1)^p\chi^p(M)=h^{p,p}(M)=b_{2p}(M)\geq1,\nonumber\ee
and hence is $\chi$-positive, where $b_{2p}(M)$ is the $2p$-th Betti number of $M$. Moreover, the odd-dimensional Betti numbers are all zero. In this case the $\chi_{-y}$-genus is essentially the Poincar\'{e} polynomial:
\be\label{4}\chi_{-y}(M)=\sum_{p=0}^nb_{2p}(M)y^p.\ee
\end{lemma}
\begin{proof}
In this case
$$\chi^p(M)=\sum_{q=0}^n(-1)^qh^{p,q}(M)=(-1)^ph^{p,p}(M).$$
By the Hodge decomposition we have $h^{p,p}(M)=b_{2p}(M)\geq 1$, and all the odd-dimensional Betti numbers vanish.
\end{proof}

Below we collect some K\"{a}hler manifolds of pure type, which are to the authors' best knowledge.
\begin{example}\label{example}
K\"{a}hler manifolds of pure type include the following examples, which are automatically $\chi$-positive.
\begin{enumerate}
\item
A rational homogeneous manifold is of the form $G/P$, where $G$ is a semisimple complex Lie group and $P$ a parabolic subgroup. It is well-known that these $G/P$ are of pure type (\cite[\S14.10]{BH}). When $G=SL(n+1,\mathbb{C})$ (the $A_n$-type), such manifolds are called complex flag manifolds (\cite[\S9]{Fu2}). In particular, for a \emph{maximal} parabolic subgroup $P_{\text{max}}\subset SL(n+1,\mathbb{C})$, $SL(n+1,\mathbb{C})/P_{\text{max}}$ are complex Grassmannians including the complex projective space $\mathbb{P}^n$.

\item
Smooth projective toric varieties are of pure type (\cite[p.106]{Fu1}).

\item
A Fano contact manifold is defined to be a complex manifold which is both complex contact (\cite[p.115]{LS}) and Fano (the first Chern class is positive or equivalently, the anti-canonical bundle is ample). Fano contact manifolds turn out to be of pure type (\cite[p.118]{LS}).

\item
Let $S$ be a projective surface and $S^{[m]}$ the Hilbert scheme of closed
0-dimensional subschemes of length $m$ on $S$. Then $S^{[m]}$ is of pure type if and only if $S$ is of pure type (\cite[\S5]{Sa}).
\end{enumerate}
\end{example}

Applying Theorem \ref{main} to K\"{a}hler manifolds of pure type yields
\begin{theorem}\label{mainapp1}
Let $M$ be a K\"{a}hler manifold of pure type. Then it satisfies optimal Chern number inequalities
\be\label{chernnumberinequality2}
\begin{split}
A_i(c_1,\ldots,c_n)[M]&\geq A_i\Big({n+1\choose 1},\ldots,{n+1\choose n}\Big)\\
&=A_{i}
(c_1,\ldots,c_n)[\mathbb{P}^n],\qquad0\leq i\leq[\frac{n}{2}],
\end{split}\nonumber\ee
where the $i$-th equality case occurs if and only if
\be\label{equalitycase1}b_{2p}(M)=h^{p,p}(M)=1,\qquad 2i\leq p\leq n.\nonumber\ee
In particular,
\begin{eqnarray}
\left\{ \begin{array}{ll}
c_n[M]\geq n+1,\\
~\\
\big[\frac{n(3n-5)}{2}c_{n}+c_{1}c_{n-1}\big][M]
\geq2(n-1)n(n+1).
\end{array} \right.\nonumber
\end{eqnarray}
\end{theorem}
\begin{remark}
Theorem \ref{mainapp1} is applicable to the manifolds mentioned in Example \ref{example}, which include rational homogeneous manifolds $G/P$. Indeed Chern numbers of $G/P$ have very specific properties. For example, among all homogenous complex manifolds (not necessarily K\"{a}hler), these $G/P$ can be characterized by the sign of their Chern numbers (\cite[Thm 2.7]{Li25}).
\end{remark}

\subsection{Signed $\chi$-positive K\"{a}hler manifolds}
Our central motivation to signed $\chi$-positivity arises from K\"{a}hler hyperbolic manifolds in the sense of Gromov, who introduced this concept in \cite{Gr} to attack the Hopf conjecture for K\"{a}hler manifolds with negative Riemannian sectional curvature.

Let $(M,\omega)$ be a K\"{a}hler manifold with $\omega$ the K\"{a}hler form, and $$(\widetilde{M},\widetilde{\omega})\overset{\pi}{\longrightarrow}(M,\omega)$$
the universal covering with $\widetilde{\omega}=\pi^{\ast}(\omega)$. $(M,\omega)$ is called \emph{K\"{a}hler hyperbolic} (\cite[p.265]{Gr}) if the two-form $\widetilde{\omega}$ is bounded on $(\widetilde{M},\widetilde{\omega})$ as a differential form. A K\"{a}hler manifold is called \emph{K\"{a}hler hyperbolic} if there exists a K\"{a}hler form on it satisfying the above-mentioned property. Clearly this definition is interesting only if the universal covering $\widetilde{M}$ is non-compact.

\begin{example}\label{example2}
Typical examples of K\"{a}hler hyperbolic manifolds include (\cite[p.265]{Gr}, \cite[\S2.2]{CY})
\begin{enumerate}
\item
K\"{a}hler manifolds which are homotopy equivalent to negatively-curved Riemannian manifolds,

\item
Compact quotients of the bounded homogeneous symmetric
domains in $\mathbb{C}^n$,

\item
submanifolds of K\"{a}hler hyperbolic manifolds, and

\item
the products of K\"{a}hler hyperbolic manifolds.
\end{enumerate}
\end{example}

The following fact is a corollary of Gromov's vanishing theorem for K\"{a}hler hyperbolic manifolds.
\begin{lemma}
K\"{a}hler hyperbolic manifolds are signed $\chi$-positive.
\end{lemma}
\begin{proof}
Here we only sketch the proof. More details and notions mentioned here can be found in \cite[\S4]{Li19} and the references therein.

Let $h^{p,q}_{(2)}(M)$ be the $L^2$-Hodge numbers of a K\"{a}hler hyperbolic manifold $M$. The $L^2$-index theorem of Atiyah (\cite[Thm 3.8]{At}) asserts that
\be\label{5}\chi^p(M)=\sum_{q=0}^n(-1)^qh^{p,q}_{(2)}(M),\ee
which indeed holds true for all complex manifolds (see \cite[(4.3)]{Li19}). Gromov's vanishing theorem for K\"{a}hler hyperbolic manifolds implies that (\cite[p.283]{Gr}) \begin{eqnarray}
\left\{ \begin{array}{ll}
h^{p,q}_{(2)}(M)=0,\qquad p+q\neq n,\\
~\\
h^{p,q}_{(2)}(M)>0,\qquad p+q=n.
\end{array} \right.\nonumber
\end{eqnarray}
This, together with (\ref{5}), leads to $$\chi^p(M)=(-1)^{n-p}h^{p,n-p}_{(2)}(M),$$ and therefore
$$(-1)^{n+p}\chi^p(M)=h^{p,n-p}_{(2)}(M)>0.$$
\end{proof}

Now we apply Theorem \ref{main} to signed $\chi$-positive K\"{a}hler manifolds to yield the following Chern number inequalities.
\begin{theorem}\label{mainapp2}
Let $M$ be a signed $\chi$-positive K\"{a}hler manifold. Then it satisfies optimal Chern number inequalities
\be\label{5.5}
\begin{split}
A_i(c_1,\ldots,c_n)[M]&\geq(-1)^nA_i\Big({n+1\choose 1},\ldots,{n+1\choose n}\Big)\\
&=(-1)^nA_{i}
(c_1,\ldots,c_n)[\mathbb{P}^n],\qquad0\leq i\leq[\frac{n}{2}],
\end{split}\ee
where the $i$-th equality case occurs if and only if
\be\label{equalitycase1}\chi^p(M)=(-1)^{n+p},\qquad 2i\leq p\leq n.\nonumber\ee
In particular,
\begin{eqnarray}\label{6}
\left\{\begin{array}{ll}
(-1)^nc_n[M]\geq n+1,\\
~\\
(-1)^n\big[\frac{n(3n-5)}{2}c_{n}+c_{1}c_{n-1}\big][M]
\geq2(n-1)n(n+1).
\end{array} \right.
\end{eqnarray}
\end{theorem}
\begin{remark}\label{remark}
\begin{enumerate}
\item
In the case of K\"{a}her hyperbolic manifolds, Theorem \ref{mainapp2} have been obtained in \cite[Thm 2.1]{Li19}, by which the signed $\chi$-positivity in Definition \ref{def} is mainly motivated.

\item
As mentioned in \cite[Thm 2.1]{Li19}, all the quality cases in (\ref{5.5}) holds if $M$ is a compact quotient of the unit ball in $\mathbb{C}^n$ with $\chi^0(M)=(-1)^n$, which is an application of Hirzebruch's proportionality principle (see \cite[\S3.3]{Li19}).

\item
A well-known conjecture, which is usually attributed to Hopf, asserts that the signed Euler characteristic of a Riemannian manifold with negatively-curved curvature is positive. The Hopf Conjecture is widely open in its generality. Gromov applied his aforementioned vanishing theorem to solve the Hopf Conjecture for K\"{a}hler manifolds in \cite{Gr}. We stress that the first inequality in (\ref{6}) is exactly an improved version of the inequality expected by the Hopf Conjecture.
\end{enumerate}
\end{remark}

It turns out that the first Chern class of a K\"{a}hler hyperbolic manifold $M$ is negative (\cite[Thm 2.11]{CY}). Hence the classical Miyaoka-Yau Chern number inequality reads (\cite[Thm 4]{Ya}):
\be\label{7}c_2(-c_1)^{n-2}[M]\geq\frac{n}{2(n+1)}(-c_1)^n[M],\ee
where the equality case in (\ref{7}) occurs if and only if $M$ is a compact quotient of the unit ball $\mathbb{C}^n$. Theorem \ref{mainapp2} says that there are more optimal Chern number inequalities for such manifolds. In particular, when $n=2$, (\ref{6}) and (\ref{7}) yield, for a K\"{a}hler hyperbolic surface $S$, that
\be\label{8}c_2[S]\geq3,\quad (c_2+c_1^2)[S]\geq12,\quad\text{and}\quad 3c_2[S]\geq c_1^2[S].\ee

Note that the first one in (\ref{8}) is redundant as it can be deduced from the last two in (\ref{8}). By Part $(2)$ in Remark \ref{remark} and the equality characterization of (\ref{7}), the equality cases in (\ref{8}) occur if $S$ is a compact quotient of the unit ball in $\mathbb{C}^2$ with $\chi^0(S)=1$. Such surfaces are precisely called \emph{fake projective planes} and have been classified by Prasad-Yeung (\cite{PY}).
These two inequalities
seems to be interesting on its own and so we record them in the following
\begin{corollary}
A K\"{a}hler hyperbolic surface $S$ satisfies two optimal Chern number inequalities
\begin{eqnarray}\left\{ \begin{array}{ll}
3c_2[S]\geq c_1^2[S]\\
~\\
(c_2+c_1^2)[S]\geq12,
\end{array} \right.\nonumber
\end{eqnarray}
where the two equality cases can be achieved by the fake projective planes.
\end{corollary}

\section{Applications to symplectic manifolds}\label{section5}
We make the convention in the sequel that all circle actions on almost-complex manifolds are nontrivial and smooth, and preserve the almost-complex structures. We usually denote by $S^1$, an $S^1$-manifold, or $M^{S^1}$ respectively the circle, a manifold equipped with a circle action, or the fixed point set of an $S^1$-manifold $M$.

In this section we shall apply Theorem \ref{main} to symplectic manifolds admitting symplectic $S^1$-action. To this end, we first recall the equivariant $\chi_{-y}$-genus for almost-complex $S^1$-manifolds and the Poincar\'{e} polynomial for symplectic $S^1$-manifolds, and then put them together to yield the desired results.
\subsection{Almost-complex $S^1$-manifolds and the $\chi_{-y}$-genus}\label{section5.1}
Assume that $M=(M,J)$ is an almost-complex $S^1$-manifold whose fixed point set $M^{S^1}$ is \emph{nonempty}. Choose an $S^1$-invariant almost-Hermitian metric on $M$. As is well-known $M^{S^1}$ consists of finitely many connected components
and each one is an almost-Hermitian submanifold of $M$. Moreover, the normal bundle of each connected component in
$M^{S^1}$ splits into a sum of complex line
bundles with respect to this $S^1$-action. Let $F\subset M^{S^1}$ be any such a connected component with complex
dimension $r$, where $r$ of course depends on the choice of $F$. As complex irreducible representations of $S^1$ are all one-dimensional, the normal bundle of $F$ in $M$, denoted by
$\nu(F)$, can be decomposed into a sum of $n-r$ complex line bundles
$$\nu(F)=\bigoplus_{i=1}^{n-r}L(F,k_i),\qquad k_i\in\mathbb{Z}-\{0\},$$
such that the action of the element $g\in S^1$ on the line bundle $L(F,k_i)$ is given by multiplying $g^{k_i}$. These $k_1,\ldots,k_{n-r}$ are usually called the \emph{weights} at $F$ with respect to this $S^1$-action. Note that these $k_i$ are
counted with multiplicities and thus are not necessarily mutually distinct. Note also that these weights are actually
independent of the almost-Hermitian metric we choose and
completely determined by the $S^1$-action. Define
\be\label{9}d_F:=\sharp\{i~|~k_i<0,~1\leq i\leq n-r\},\ee
i.e., $d_F$ is the number of negative weights at $F$.

With these notions understood, we have the following localization formula for the $\chi_{-y}$-genus, which is essentially due to Kosniowski (\cite{Ko70}).
\begin{theorem}
Let $M$ be an almost-complex $S^1$-manifold with $M^{S^1}=\coprod F$. Then
\be\label{chiylocalization}
\chi_{-y}(M)=\sum_F\chi_{-y}(F)\cdot y^{d_F},\ee
where the sum is over the connected components $F$ in $M^{S^1}$ and $d_F$ given by (\ref{9}).
\end{theorem}
\begin{remark}
The proof of this formula is a typical application of the Atiyah-Bott-Singer fixed point theorem (see also \cite[\S5.7]{HBJ}). The first author refined this idea in \cite{Li12} to give some related applications in symplectic geometry.
\end{remark}

\subsection{Symplectic and Hamiltonian $S^1$-actions}
We assume in this subsection that $M=(M,\omega)$ is a symplectic manifold.

An $S^1$-action on $M$ is called \emph{symplectic} if it preserves the symplectic form $\omega$: $g^{\ast}(\omega)=\omega$ for any $g\in S^1$. In such case $M=(M,\omega)$ is called a \emph{symplectic $S^1$-manifold}. If $(M,\omega)$ is a symplectic $S^1$-manifold, it is well-known that we can always find an almost-complex structure both compatible with $\omega$ and preserved by this $S^1$-action. So notions in Section \ref{section5.1} can be applicable to the setting of symplectic $S^1$-actions \emph{without} explicitly mentioning this compatible almost-complex structure.

Let $X$ be the generating vector field of an $S^1$-action on $(M,\omega)$. This action is symplectic if and only if the one-form $\omega(X,\cdot)$ is closed, which is due to the Cartan formula (\cite[p.71]{Au}). A symplectic $S^1$-action on $(M,\omega)$ is called \emph{Hamiltonian} if the one-form $\omega(X,\cdot)$ is exact, i.e., $\omega(X,\cdot)=\text{d}f$ for some smooth function $f$ on $M$. In such case $(M,\omega)$ is called a \emph{Hamiltonian $S^1$-manifold}. This $f$ is usually called the \emph{moment map} of this Hamiltonian $S^1$-action, which is unique up to an additive constant. For a Hamiltonian $S^1$-action on $M$, the set $M^{S^1}$ is exactly that of the critical points of the moment map $f$ and hence \emph{nonempty} as the points minimizing and maximizing $f$ are critical. Nevertheless, in general a symplectic $S^1$-action on $M$ with nonempty $M^{S^1}$ may not be Hamiltonian, even if $M^{S^1}$ only consist of isolated fixed points (\cite{To}).

Let us at this moment digress to briefly recall a notion introduced by Novikov (see \cite{No} or \cite[\S1.5]{Fa2}). For a finite CW-complex $X$, any cohomology class $\xi\in H^1(X;\mathbb{R})$ can be associated to a sequence of nonnegative integers $b_i(\xi)$ ($0\leq i\leq\dim X$), now known as the \emph{Novikov numbers}, which are analogous to and bounded above by the usual Betti numbers $b_i(X)$ (\cite[\S1.6]{Fa2}). The precise definition of Novikov numbers is not important in our article but only the following fact is needed: $b_i(0)=b_i(X)$ (\cite[Prop.1.28]{Fa2}).

Suppose now that $M$ is equipped with a symplectic $S^1$-action, $X$ its generating vector field, and $\xi:=[\omega(X,\cdot)]\in H^1(M;\mathbb{R})$ the de Rham cohomology class of the closed one-form $\omega(X,\cdot)$. Thus this symplectic $S^1$-manifold can be attached to the Novikov numbers $b_i(\xi)$, which reduce to the usual Betti numbers $b_i(M)$ whenever this symplectic $S^1$-action is Hamiltonian.

For a symplectic $S^1$-manifold $M$, the associated Novikov numbers $b_i(\xi)$ can be calculated in terms of the information around the fixed point set $M^{S^1}$ as follows (\cite[Thm 7.5]{Fa2}).
\begin{theorem}\label{farberthm}
Let $(M,\omega)$ be a symplectic $S^1$-manifold with $M^{S^1}=\coprod F$, $X$ the generating vector field and $\xi=[\omega(X,\cdot)]\in H^1(M;\mathbb{R})$. Then
\be\label{Morselocation}\sum_{i=0}^{2n}b_i(\xi)y^i=\sum_FP_y(F)y^{2d_F},\ee
where $$P_y(F):=\sum_{j=0}^{\dim F}b_j(F)y^j$$ is the Poincar\'{e} polynomial of $F$, $d_F$ introduced in (\ref{9}), and the sum over the connected components $F$ in $M^{S^1}$.
\end{theorem}
\begin{remark}
When the action in Theorem \ref{farberthm} is Hamiltonian, the formula (\ref{Morselocation}) is well-known as in this case the moment map is a perfect Morse-Bott function and the Morse-Bott index of the critical submanifold $F$ is exactly $2d_F$ (see \cite{Ki}, \cite[p.108]{Au}, \cite{PR}). For K\"{a}hler manifolds this is due to Frankel (\cite{Fr59}), building on the pioneering work of Bott (\cite{Bo}).
\end{remark}

\subsection{Applications to symplectic $S^1$-manifolds}
The following fact provides examples for $\chi$-positive symplectic manifolds.
\begin{lemma}\label{lemma2}
Let $(M,\omega)$ be a symplectic $S^1$-manifolds with isolated fixed points such that all the associated even-dimensional Novikov numbers $b_{2i}(\xi)$ ($0\leq i\leq n$) are nonzero. Then $M$ is $\chi$-positive. In particular, all Hamiltonian $S^1$-manifolds with isolated fixed points are $\chi$-positive.
\end{lemma}
\begin{proof}
If the fixed point set $M^{S^1}=\coprod F$ consists of isolated fixed points $F$, then (\ref{chiylocalization}) and (\ref{Morselocation}) read
$$\chi_{-y}(M)=\sum_Fy^{d_F}$$
and
$$\sum_{i=0}^nb_{2i}(\xi)y^i=\sum_Fy^{d_F}$$
respectively, which imply that
\be\label{3}\chi_{-y}(M)=\sum_{i=0}^nb_{2i}(\xi)y^i\ee
and hence $M$ is $\chi$-positive if all even-dimensional Novikov numbers $b_{2i}(\xi)$ are nonzero. When the action is Hamiltonian, $b_{2i}(\xi)$ are the Betti numbers $b_{2i}(M)$, which are nonzero as $0\neq[\omega^i]\in H^{2i}(M;\mathbb{R})$.
\end{proof}
\begin{remark}\label{remarkoddvanishing}
All the odd-dimensional Novikov numbers (resp. Betti numbers) of symplectic (resp. Hamiltonian) $S^1$-manifolds with isolated fixed points are necessarily zero, still due to (\ref{Morselocation}). Comparing (\ref{3}) with (\ref{4}), it is interesting to see that Hamiltonian $S^1$-manifolds with isolated fixed points behave like K\"{a}hler manifolds of pure type.
\end{remark}

Combining Theorem \ref{main} with Lemma \ref{lemma2} yields
\begin{theorem}\label{mainapp3}
Let $M$ be a symplectic $S^1$-manifold with isolated fixed points and all the associated even-dimensional Novikov numbers $b_{2p}(\xi)$ are nonzero. Then $M$ satisfies optimal Chern number inequalities
\be\label{chernnumberinequality2}
\begin{split}
A_i(c_1,\ldots,c_n)[M]&\geq A_i\Big({n+1\choose 1},\ldots,{n+1\choose n}\Big)\\
&=A_{i}
(c_1,\ldots,c_n)[\mathbb{P}^n],\qquad0\leq i\leq[\frac{n}{2}],
\end{split}\nonumber\ee
where the $i$-th equality case occurs if and only if
\be\label{equalitycase1}b_{2p}(\xi)=1,\qquad 2i\leq p\leq n.\nonumber\ee
In particular,
\begin{eqnarray}
\left\{ \begin{array}{ll}
c_n[M]\geq n+1,\\
~\\
\big[\frac{n(3n-5)}{2}c_{n}+c_{1}c_{n-1}\big][M]
\geq2(n-1)n(n+1).
\end{array} \right.\nonumber
\end{eqnarray}
Moreover, these results particularly hold true for Hamiltonian $S^1$-manifolds with isolated fixed points.
\end{theorem}

\section{The signature formula on symplectic $S^1$-manifolds}\label{section6}
We derive in this section a signature formula (Theorem \ref{mainapp4}) for symplectic $S^1$-manifolds, which is inspired by the work of M. Farber (\cite{Fa1}).
\subsection{The signature}
Denote by $\sigma(X)$ the \emph{signature} of a $4m$-dimensional manifold $X$. By definition $\sigma(X)=b^{+}_{2m}-b^-_{2m}$, where $b^{+}_{2m}$ (resp. $b^-_{2m}$) is the dimension of maximal subspace in $H^{2m}(X;\mathbb{R})$ where the intersection pairing is positive-definite (resp. negative-definite). We have $b^{+}_{2m}+b^-_{2m}=b_{2m}(X)$ due to the Poincar\'{e} duality. The convention that $\sigma(X)=0$ is understood if the dimension $X$ is not divisible by $4$.

The Hirzebruch signature theorem says that $\sigma(X)$ can be expressed as a specific rational linear combination of Pontrjagin numbers via the $L$-genus introduced by him (\cite{Hi66}). Beautiful closed formulas in terms of some variant of multiple zeta values have been given in \cite{BB} for these coefficients, as well as those of the $\hat{A}$-genus. When the manifold $X$ in consideration is (stably) almost-complex, $\sigma(X)$ is then a specific rational linear combination of Chern numbers, of whose coefficients the closed formulas are given in \cite[Thm 2.4]{LL}.

As mentioned in Section \ref{section2.2}, for an almost-complex manifold $M$ we have the identity $\chi_y(M)\big|_{y=1}=\sigma(M)$, even if $2n$, the real dimension of $M$, is not divisible by $4$. Indeed, by putting $y=1$ in (\ref{chiyduality}) we have $\chi_y(M)\big|_{y=1}=0$ whenever $n$ is odd.

\subsection{The signature of symplectic $S^1$-manifolds}
Before stating the result, we give the following
\begin{definition}\label{def2}
An even-dimensional manifold $X$ is called \emph{signature-alternating} if \be\label{sig alter}\sigma(X)=\sum_{i=0}^{\frac{1}{2}\dim X}(-1)^ib_{2i}(X),\ee
i.e., $\sigma(X)$ is equal to the alternating sum of even-dimensional Betti numbers of $X$.
\end{definition}
\begin{remark}\label{remark2}
When $\dim X\equiv 2\pmod 4$, the right-hand side of (\ref{sig alter}) is zero due to the Poincar\'{e} duality, and thus (\ref{sig alter}) still trivially holds true in this case. So (\ref{sig alter}) is interesting only if $\dim X\equiv0\pmod 4$. Notice that (\ref{sig alter}) is also trivially true when $X$ is a point.
\end{remark}

With Definition \ref{def2} in hand, the main observation in this section is the following result, which is inspired by and meanwhile improve on a result of Farber (see \cite[p.210]{Fa1} or \cite[Thm 7.10]{Fa2})
\begin{theorem}\label{mainapp4}
Let $M$ be a symplectic $S^1$-manifold with $M^{S^1}=\coprod F$, $X$ the generating vector field and $\xi=[\omega(X,\cdot)]\in H^1(M;\mathbb{R})$. If all the connected components $F$ in $M^{S^1}$ are signature alternating, then
\be\label{10}\sigma(M)=\sum_{i=0}^n(-1)^ib_{2i}(\xi).\ee
\end{theorem}
\begin{remark}
\begin{enumerate}
\item
Farber showed Theorem \ref{mainapp4} under the additional requirement that the odd-dimensional Betti numbers of $F$ are all zero. A manifold both satisfies (\ref{sig alter}) and with vanishing odd-dimensional Betti numbers is called an \emph{$i$-manifold} in \cite{Fa1} and \cite{Fa2}.

\item
Theorem \ref{mainapp4} implies that, if all the connected components in $M^{S^1}$ of a Hamiltonian $S^1$-manifold $M$ are signature-alternating, then so is $M$.
\end{enumerate}
\end{remark}

With Remark \ref{remark2} in mind, Theorem \ref{mainapp4} has the following consequence.
\begin{corollary}\label{coro}
Let $M$ be a symplectic $S^1$-manifold with $M^{S^1}=\coprod F$. If each connected component $F$ in $M^{S^1}$ is either an isolated point or of $\dim F\equiv 2\pmod 4$, then
\be\label{11}\sigma(M)=\sum_{i=0}^n(-1)^ib_{2i}(\xi).\nonumber\ee
\end{corollary}
\begin{remark}
When the action is Hamiltonian with only isolated fixed points, Corollary \ref{coro} is due to Jones-Rawnsley (\cite{JR}). See \cite[p.209]{Fa1} or \cite[Thm 7.9]{Fa2} for the symplectic $S^1$-manifolds with isolated fixed points.  See also \cite[Thm 2.2]{Lin1} for an equivalent form of this corollary when the action is Hamiltonian.
\end{remark}

\subsection{Proof of Theorem \ref{mainapp4}}
\begin{proof}
Let the notation be as in Theorem \ref{mainapp4}. For each connected component $F$ define
$$P_y^{\text{even}}(F):=\sum_{i=0}^{\frac{1}{2}\dim F}b_{2i}(F)y^i.$$
Note that $F$ being signature-alternating is equivalent to $\sigma(F)=P_{-1}^{\text{even}}(F)$. Applying this notation we consider only the even powers on both sides of (\ref{Morselocation}), which leads to
\be\sum_{i=0}^{n}b_{2i}(\xi)y^i=
\sum_{F}P_y^{\text{even}}(F)y^{d_F}\nonumber\ee
and hence
\be\label{12}
\begin{split}
\sum_{i=0}^{n}b_{2i}(\xi)(-1)^i&=
\sum_{F}P_{-1}^{\text{even}}(F)(-1)^{d_F}\\
&=\sum_{F}\sigma(F)(-1)^{d_F}.\quad(\text{$F$ being signature-alternating})
\end{split}\ee

On the other hand, specializing to $y=-1$ in (\ref{chiylocalization}) tells us that
\be\label{13}\sigma(M)=\sum_{F}\sigma(F)(-1)^{d_F}.\ee
Combining (\ref{12}) with (\ref{13}) yields the desired (\ref{10}).
\end{proof}

\section{Betti number restrictions on Hamiltonian $S^1$-manifolds}\label{section7}
In this last section we shall apply the signature formula in Section \ref{section6} to prove Betti numbers restrictions on Hamiltonian $S^1$-manifolds. To put our applications into perspective, we first recall some background results.
\subsection{Background results}
The classical Hodge theory imposes strong restrictions on the underlying topology of K\"{a}hler manifolds. For instance, the Hard Lefschetz theorem (\cite[p.122]{GH}) assert that, for a K\"{a}hler manifold $(M,\omega)$, we have the isomorphisms:
$$\wedge[\omega]^{n-i}:\quad H^i(M;\mathbb{R})\xrightarrow{\cong}
H^{2n-i}(M;\mathbb{R}),\quad0\leq i\leq n-1.$$
This particularly implies that
$$\wedge[\omega]:\quad H^i(M;\mathbb{R})\xrightarrow{\cong}
H^{i+2}(M;\mathbb{R}),\quad0\leq i\leq n-2$$
is injective, and hence the even-dimensional or odd-dimensional Betti numbers are \emph{unimodal}:
$$b_i(M)\leq b_{i+2}(M),\quad 0\leq i\leq n-2.$$

Recall from Remark \ref{remarkoddvanishing} that only even-dimensional Betti numbers are involved in Hamiltonian $S^1$-manifolds with isolated fixed points, and there exists some similarity between K\"{a}hler manifold of pure type and Hamiltonian $S^1$-manifolds with isolated fixed points. Therefore the following question posed by Tolman seems to be natural (see \cite[p.11]{JHKLM}).
\begin{question}\label{question}
Let $M$ be a Hamiltonian $S^1$-manifold with isolated fixed points. Is the sequence of inequalities
\begin{eqnarray}
\left\{ \begin{array}{ll}
b_2(M)\leq b_4(M)\cdots\leq b_{n-2}(M)\leq b_n(M)\quad&(\text{$n$ even})\\
~\\
b_2(M)\leq b_4(M)\cdots\leq b_{n-3}(M)\leq b_{n-1}(M)\quad&(\text{$n$ odd})
\end{array} \right.\nonumber
\end{eqnarray}
true?
\end{question}
Cho-Kim answered Question \ref{question} affirmatively when $\dim M=8$ (\cite[Thm 1.2]{CK}). Cho improved the result of \cite{CK} by showing the following inequality in \cite[Thm 2]{Ch} when $M^{S^1}$ are isolated, which was further extended by Lindsay (\cite[Prop.5.4]{Lin2}) in the following form.
\begin{theorem}[Cho, Lindsay]\label{theorem of Cho}
Let $M$ be a Hamiltonian $S^1$-manifold and $dim M=8k$ or $8k+4$. Assume that any connected component $F$ in $M^{S^1}$ is either an isolated point or of $\dim F\equiv2 \pmod4$.
Then
\be\sum_{i=0}^{k-1}b_{2+4i}(M)\leq\sum_{i=0}^{k-1}b_{4+4i}(M).\nonumber\ee
In particular, $b_2(M)\leq b_4(M)$ when $\dim M=8$ or $12$.
\end{theorem}

\subsection{Main results in this section}
In order to state our main results in this section, let us give the following definition, which was introduced in \cite[Def 1.1]{Li13} by the first author for K\"{a}her manifolds.
\begin{definition}\label{def for cauchy-schwarz}
Let $(M,\omega)$ be a symplectic manifold. $H^{2i}(M;\mathbb{R})$ ($1\leq i\leq[n/2]$) is said to satisfy the \emph{reverse Cauchy-Schwarz} (resp. \emph{Cauchy-Schwarz}) inequality if for any $\alpha\in H^{2i}(M;\mathbb{R})$ we have
\be\label{formula for cauchy-schwarz}(\int_M\alpha\wedge[\omega]^{n-i})^2\geq
(\int_M\alpha^2\wedge[\omega]^{n-2i})(\int_M[\omega]^{n}),\nonumber\ee
$$\Big(\text{resp. $(\int_M\alpha\wedge[\omega]^{n-i})^2\leq
(\int_M\alpha^2\wedge[\omega]^{n-2i})(\int_M[\omega]^{n})$},\Big)$$
and the equality case occurs if and only if $\alpha$ is proportional to $[\omega^i]$, i.e.,  $\alpha\in\mathbb{R}[\omega^i].$
\end{definition}
\begin{remark}\label{remark3}
In the case of K\"{a}hler manifolds, the first author gives in \cite[Thm 1.3]{Li13} a sufficient and necessary condition in terms of Hodge numbers to characterize when the (reverse) Cauchy-Schwarz inequality holds true on $H^{2i}(M;\mathbb{R})$, and extends it to the mixed version in \cite[Thm 1.3]{Li16}. The main tools in \cite{Li13} and \cite{Li16} are the classical and mixed Hodge-Riemann bilinear relations respectively.
\end{remark}

With Definition \ref{def for cauchy-schwarz} understood, we have the following results, whose first part extends Theorem \ref{theorem of Cho} by relaxing the assumption and meanwhile characterizing the equality case.
\begin{theorem}\label{mainapp5}
\begin{enumerate}
\item
Let $(M,\omega)$ be a signature-alternating symplectic manifold with $\dim M=8k$ or $8k+4$. Then \be\label{16.1}\sum_{i=0}^{k-1}b_{2+4i}(M)\leq\sum_{i=0}^{k-1}b_{4+4i}(M),\ee
and the equality case in (\ref{16.1}) occurs if and only if $H^{4k}(M;\mathbb{R})$ or $H^{4k+2}(M;\mathbb{R})$ satisfies the reverse Cauchy-Schwarz inequality.

\item
Let $(M,\omega)$ be a signature-alternating symplectic manifold with $\dim M=8k$ or $8k-4$. Then \be\label{16.2}\sum_{i=0}^{k-1}b_{2+4i}(M)\geq\sum_{i=0}^{k-1}b_{4i}(M),\ee
and the equality case in (\ref{16.2}) occurs if and only if $H^{4k}(M;\mathbb{R})$ or $H^{4k-2}(M;\mathbb{R})$ satisfies the Cauchy-Schwarz inequality.
\end{enumerate}
\end{theorem}

By Theorem \ref{mainapp4} and Corollary \ref{coro} we have
\begin{corollary}\label{coro2}
Let $(M,\omega)$ be a Hamiltonian $S^1$-manifold with $M^{S^1}=\coprod F$. Assume that all the connected components $F$ are signature-alternating, which particularly hold true if these $F$ are either isolated points or of $\dim F\equiv 2 \pmod 4$.
\begin{enumerate}
\item
If $\dim M=8k$ or $8k+4$, the inequality (\ref{16.1}) holds true, and the equality case in (\ref{16.1}) occurs if and only if $H^{4k}(M;\mathbb{R})$ or $H^{4k+2}(M;\mathbb{R})$ satisfies the reverse Cauchy-Schwarz inequality.

\item
If $\dim M=8k$ or $8k-4$, the inequality (\ref{16.2}) holds true, and the equality case in (\ref{16.2}) occurs if and only if $H^{4k}(M;\mathbb{R})$ or $H^{4k-2}(M;\mathbb{R})$ satisfies the Cauchy-Schwarz inequality.
\end{enumerate}
\end{corollary}
\begin{remark}
As mentioned before the inequality (\ref{16.1}) in Corollary \ref{coro2} was obtained in \cite[Prop.5.4]{Lin2}, but without the characterization of the equality case .
\end{remark}

Taking $k=1$ (resp. $k=2$) in Part $(1)$ \big(resp. Part $(2)$\big) in Corollary \ref{coro2} yields the following consequences, which provide positive evidence towards Question \ref{question} under even more flexible conditions.
\begin{corollary}
Let $(M,\omega)$ be a Hamiltonian $S^1$-manifold with $M^{S^1}=\coprod F$, and these $F$ are either isolated points or of $\dim F\equiv 2 \pmod 4$.
\begin{enumerate}
\item
When $\dim M=8$ or $12$, we have $b_2(M)\leq b_4(M)$, with equality if and only if $H^{4}(M;\mathbb{R})$ or $H^{6}(M;\mathbb{R})$ satisfies the reverse Cauchy-Schwarz inequality.

\item
When $\dim M=12$ or $16$, and $b_2(M)=1$, we have $b_4(M)\leq b_6(M)$, with equality if and only if $H^{6}(M;\mathbb{R})$ or $H^{8}(M;\mathbb{R})$ satisfies the Cauchy-Schwarz inequality.
\end{enumerate}
\end{corollary}

\subsection{Proof of Theorem \ref{mainapp5}}
As mentioned in Remark \ref{remark3}, in \cite{Li13} for K\"{a}hler manifolds a sufficient and necessary condition in terms of Hodge numbers is given to characterize precisely when the (reverse) Cauchy-Schwarz inequality holds on $H^{2i}(M;\mathbb{R})$, thanks to the Hodge theory. In the absence of such a theory for general symplectic manifolds, it seems difficult to present a condition in terms of suitable invariants on them to characterize (reverse) Cauchy-Schwarz inequality for \emph{all} $1\leq i\leq [n/2]$, which we pose at the end of this article as an open question (Question \ref{question2}). Nevertheless, we have the following solution when the complex dimension $n$ is even and $i=n/2$.
\begin{proposition}\label{prop}
For a $4m$-dimensional symplectic manifold $(M,\omega)$, $H^{2m}(M;\mathbb{R})$ satisfies the reverse Cauchy-Schwarz (resp. Cauchy-Schwarz) inequality if and only if $b^+_{2m}(M)=1$ (resp. $b^-_{2m}(M)=0$).
\end{proposition}
\begin{proof}
First note that we always have $b^+_{2m}(M)\geq 1$ as $$\int_M[\omega^m]\wedge[\omega^m]=\int_M[\omega^{2m}]>0.$$
Take for simplicity $b^{\pm}:=b^{\pm}_{2m}(M)$.

Assume that $b^+=1$. We choose a base $$\alpha_1=\frac{[\omega]^{m}}{(\int_M[\omega]^{2m})^{1/2}},
~\beta_1,~\ldots,~\beta_{b^-}$$
of $H^{2m}(M;\mathbb{R})$ such that with respect to it the matrix of the intersection pairing $\int_M(\cdot)\wedge(\cdot)$ on $H^{2m}(M;\mathbb{R})$ is
$\text{diag$(1,-1,\ldots,-1)$}$.

Write $$\alpha=\lambda\alpha_1+\lambda_1\beta_1+\cdots+\lambda_{b^-}\beta_{b^-},\quad
(\lambda,\lambda_i\in\mathbb{R}),$$
We have
$$(\int_M\alpha\wedge[\omega]^{m})^2=\lambda^2\int_M[\omega]^{2m}$$
and
$$(\int_M\alpha^2)(\int_M[\omega]^{2m})=(\lambda^2-\sum_{i=1}^{b^-}\lambda_i^2)\int_M[\omega]^{2m},$$
and hence
$$(\int_M\alpha\wedge[\omega]^{m})^2-(\int_M\alpha^2)(\int_M[\omega]^{2m})=
(\sum_{i=1}^{b^-}\lambda_i^2)\int_M[\omega]^{2m}\geq0,$$
with equality if and only if all $\lambda_i=0$, i.e., $\alpha=\lambda\alpha_1$ is proportional to $[\omega]^{m}$.

Conversely, assume that $H^{2m}(M;\mathbb{R})$ satisfies the reverse Cauchy-Schwarz inequality and choose a base
$$\alpha_1=\frac{[\omega]^{m}}{(\int_M[\omega]^{2m})^{1/2}},~\alpha_2,~\ldots,
~\alpha_{b^+},
~\beta_1,~\ldots,~\beta_{b^-}$$
such that the matrix with respect to it is
\be\label{17}\text{diag$(\underbrace{1,\ldots,1}_{b^+},\underbrace{-1,\ldots,-1}_{b^-})$}.\ee
If $b^+\geq 2$, then
$$(\int_M\alpha_2\wedge[\omega]^{m})^2-(\int_M\alpha^2_2)(\int_M[\omega]^{2m})
=-\int_M[\omega]^{2m}<0,$$
which is a contradiction.

If $b^-=0$, we may choose a base
$$\alpha_1=\frac{[\omega]^{m}}{(\int_M[\omega]^{2m})^{1/2}},~\alpha_2,~\ldots,
~\alpha_{b^+},$$ whose corresponding matrix is $\text{diag$(1,\ldots,1)$}$. Then similar arguments as above yields the validity of the Cauchy-Schwarz inequality on $H^{2m}(M;\mathbb{R})$.

At last assume that $H^{2m}(M;\mathbb{R})$ satisfies the Cauchy-Schwarz inequality and on the contrary that $b^-\geq1$. Choose a base
$$\alpha_1=\frac{[\omega]^{m}}{(\int_M[\omega]^{2m})^{1/2}},~\alpha_2,~\ldots,
~\alpha_{b^+},~\beta_{1},~\ldots,~\beta_{b^-},$$ whose corresponding matrix is (\ref{17}). Then it is easy to see that the element $\beta_1$ contradicts to the Cauchy-Schwarz inequality:
$$(\int_M\beta_1\wedge[\omega]^{m})^2=0>
(\int_M\beta_1^2)(\int_M[\omega]^{2m})=-\int_M[\omega]^{2m}.$$
\end{proof}

Now we are ready to prove Theorem \ref{mainapp5}.
\begin{proof}
Assume that $(M,\omega)$ is a $4m$-dimensional symplectic manifold which is signature-alternating. For simplicity we take $b_i:=b_i(M)$ and $b^{\pm}_{2m}:=b^{\pm}_{2m}(M)$. Recall that $$\sigma(M)=b^+_{2m}-b^-_{2m},\quad b_{2m}=b^+_{2m}+b^-_{2m},\quad b^+_{2m}\geq1.$$
Hence we have

\begin{eqnarray}
\begin{split}
2&\leq 2b^+_{2m}\\
&=b_{2m}+\sigma(M)\\
&=b_{2m}+\big[b_0-b_2+b_4+\cdots+(-1)^mb_{2m}+\cdots-b_{4m-2}+b_{4m}\big]\\
&=\left\{\begin{array}{ll}
2+2(-b_2+b_4-b_6+\cdots-b_{2m-2}+b_{2m})\qquad\text{($m$ even)}\\
~\\
2+2(-b_2+b_4-b_6+\cdots+b_{2m-2}).\qquad\text{($m$ odd)}
\end{array}\right.\nonumber
\end{split}
\end{eqnarray}
This implies that
\begin{eqnarray}\label{18}
\left\{ \begin{array}{ll}
b_2+b_6+\cdots+b_{2m-2}\leq b_4+b_8+\cdots+b_{2m}\qquad\text{($m$ even, $m\geq2$)}\\
~\\
b_2+b_6+\cdots+b_{2m-4}\leq b_4+b_8+\cdots+b_{2m-2},\qquad\text{($m$ odd, $m\geq3$)}
\end{array} \right.
\end{eqnarray}
and the equality case occurs if and only if $b^+_{2m}=1$.

Taking $m=2k$ or $2k+1$ respectively in (\ref{18}) leads to the inequality (\ref{16.1}), and the equality characterization follows from Proposition \ref{prop}. This completes the proof of the first part in Theorem \ref{mainapp5}.

The treatment of the second part is similar because
\begin{eqnarray}\label{19}
\begin{split}
0&\leq 2b^-_{2m}\\
&=b_{2m}-\sigma(M)\\
&=b_{2m}-\big[b_0-b_2+b_4+\cdots+(-1)^mb_{2m}+\cdots-b_{4m-2}+b_{4m}\big]\\
&=\left\{\begin{array}{ll}
2(b_2-b_0+b_6-b_4+\cdots+b_{2m-2}-b_{2m-4})\qquad\text{($m$ even)}\\
~\\
2(b_2-b_0+b_6-b_4+\cdots+b_{2m}-b_{2m-2}),\qquad\text{($m$ odd)}
\end{array}\right.
\end{split}
\end{eqnarray}
with equality if and only if $b^-_{2m}=0$. Taking $m=2k$ or $2k-1$ respectively in (\ref{19}) yields the inequality (\ref{16.2}), and the equality characterization also follows from Proposition \ref{prop}.
\end{proof}

Let us end this article by posing the following question, which we think may be interesting on its own.
\begin{question}\label{question2}
Let $(M,\omega)$ be a symplectic manifold. For each $1\leq i\leq [n/2]$, can we give a sufficient and necessary condition in terms of suitable invariants of $(M,\omega)$ to characterize precisely when the (reverse) Cauchy-Schwarz inequality introduced in Definition \ref{def for cauchy-schwarz} holds true on $H^{2i}(M;\mathbb{R})$ ?
\end{question}
\begin{remark}
The very special case of $n$ being even and $i=n/2$ has been answered in Proposition \ref{prop}.
\end{remark}

\end{document}